\documentclass[11pt,twoside]{article}

\usepackage{amssymb,amsmath,amsthm,latexsym}
\usepackage{amsfonts}
\usepackage{amsfonts}
\usepackage{graphicx}

\newtheorem{theorem}{Theorem}[section]

\newtheorem{conjecture}[theorem]{Conjecture}

\newtheorem{lemma} [theorem]{Lemma}

\newtheorem{remark}[theorem]{Remark}

\newcommand{\op}{\overline{p}}
\newcommand{\opt}{\overline{OPT}}

\vspace{5cm}

\begin{document}
	
	\label{'ubf'}  
	\setcounter{page}{1}                                 %Put here the starting page number

	\markboth {\hspace*{-9mm} \centerline{\footnotesize \sc
			% Put here the left page top label 
		Arithmetic Properties for Overpartition $k$-tuples with Odd Parts}
	}
	{ \centerline                           {\footnotesize \sc  
			%put here the author's name
		Manjil P. Saikia, Abhishek Sarma and James A. Sellers} \hspace*{-9mm}              
	}

	\vspace*{-2cm}
	%\begin{flushleft}
	%	{\footnotesize\it J. of the Ramanujan Mathematical Society,  XX, No. XX (XXXX),  pp. XX--XX \\ %\pageref{'ubf'}-\pageref{'ubl'}\\  
			% Use the appropriate labels
	%	}
	%	\vspace*{0.3cm}       %{2.2cm}
%	\end{flushleft}

\begin{center}
		{
			{\Large \textbf{Arithmetic Properties Modulo Powers of 2 for Overpartition $k$-Tuples with Odd Parts}
			}

			\medskip

			{\sc Manjil P. Saikia}\\
			{\footnotesize Mathematical and Physical Sciences division, School of Arts and Sciences, Ahmedabad University, Ahmedabad, Gujarat, PIN 380009, India}\\
			{\footnotesize e-mail: {\it manjil@saikia.in}}\\
			{\sc Abhishek Sarma}\\
			{\footnotesize Department of Mathematical Sciences, Napaam, Tezpur University, Tezpur, Assam, PIN 784028, India.}\\
			{\footnotesize e-mail: {\it abhitezu002@gmail.com}}\\
                {\sc James A. Sellers}\\
			{\footnotesize {Department of Mathematics and Statistics, University of Minnesota Duluth, Duluth, MN 55812, USA}}\\
			{\footnotesize e-mail: {\it jsellers@d.umn.edu}}
}\end{center}

\author{Saikia, Manjil P. \& Sarma, Abhishek \& Sellers, James A.}

	\thispagestyle{empty}

	\hrulefill

	\begin{abstract}  
        Recently, Drema and Saikia (2023) proved several congruences modulo powers of $2$ and $3$ for overpartition triples with odd parts. We extend their list substantially. We prove several congruences modulo powers of $2$ for overpartition $k$-tuples with odd parts, along with a few infinite families of congruences for overpartition triples with odd parts and overpartition $k$-tuples with odd parts (for $k=4$ and for odd $k$).
	\end{abstract}
	\hrulefill



\section{Introduction}
A partition of a positive integer $n$ is a finite non-increasing sequence of positive integers $\lambda=(\lambda_1, \lambda_2, \ldots, \lambda_k)$ such that the parts $\lambda_i$ sum up to $n$. For instance, $4,3+1,2+2,2+1+1$ and $1+1+1+1$ are the five partitions of $4$. The number of partitions of $n$ is denoted by $p(n)$, and its generating function found by Euler is given by
\[
\sum_{n\geq 0}p(n)q^n=\frac{1}{(q;q)_\infty},
\]
where
\[
(a;q)_\infty:=\prod_{i\geq 0}(1-aq^i), \quad |q|<1.
\] Throughout the paper, we will use the notation $f_{k} := (q^k;q^k)_{\infty}$.

Among the many beautiful properties satisfied by the partition function, undoubtedly some of the most beautiful are the congruences modulo primes that the partition function satisfies, as discovered by Ramanujan. Several mathematicians have since studied the arithmetic properties of the partition function, as well as other generalized classes of partitions. Recently, Drema and Saikia \cite{DremaSaikia} studied congruences modulo small powers of $2$ and $3$ satisfied by the function which counts overpartition triples with odd parts. The aim of the present paper is to extend their list and to also look at some more general cases.

An overpartition of a nonnegative integer $n$ is a non-increasing sequence of natural numbers whose sum is $n$, where the first occurrence (or equivalently, the last occurrence) of a number may be overlined. The eight overpartitions of 3 are
\[3,\bar{3},2+1,\bar{2}+1,2+\bar{1},\bar{2}+\bar{1},1+1+1,~\text{and}~\bar{1}+1+1.\]The number of overpartitions of $n$ is denoted by $\op(n)$ and its generating function is given by
\[
\sum_{n\geq 0}\op(n)q^n=\frac{f_{2}}{f_{1}^2}.
\]
An overpartition triple $\xi$ of $n$ is a triplet $(\xi_1, \xi_2, \xi_3)$ of overpartitions such that the sum of all the parts is $n$. For instance $(\bar 3 +2, \bar 5, 1+1+1+1+1)$ is an overpartition triple of $15$. An overpartition triple into odd parts is a triplet of overpartitions $(\xi_1, \xi_2, \xi_3)$ such that the parts of all the overpartitions $\xi_1, \xi_2$ and $\xi_3$ are restricted to odd integers. For instance $(\bar 3 +\bar 1 + 1, \bar 5, \bar 1+1+1+1+1)$ is an overpartition triple into odd parts of $15$. Let $\opt(n)$ denote the number of such overpartition triples of $n$ with odd parts.  Then the generating function for $\opt(n)$ is given by \cite[Eq. (1)]{DremaSaikia}
\begin{equation}
    \sum_{n\geq 0}\opt(n)q^n=\frac{f_{2}^9}{f_{1}^6f_{4}^3}=\prod_{i\geq 0}\left(\frac{1+q^{2i+1}}{1-q^{2i+1}}\right)^3. \label{opt}
\end{equation}

Before proceeding further, we mention (without commentary) the following easy consequence of the binomial theorem which will be used below.
 \begin{lemma}\cite[Lemma 2.6]{NayakaNaika}
      For a prime $p$, and positive integers $k$ and $l$, we have
     \begin{align*}
         f_{k}^{p^l} \equiv f_{pk}^{p^{l-1}} \pmod{p^l}.
     \end{align*}
 \end{lemma}

It can be shown easily \cite[Eq. (43)]{DremaSaikia} that
\[
\opt(2n+1)\equiv 0 \pmod 2,
\]
for all $n\geq 0$. 

Indeed, working modulo 2, we have
\begin{align*}
     \sum_{n\geq 0}\opt(n)q^n \equiv  
     \prod_{i\geq 0}\left(\frac{1+q^{2i+1}}{1-q^{2i+1}}\right)^3 \equiv \prod_{i\geq 0}\left(\frac{1-q^{2i+1}}{1-q^{2i+1}}\right)^3 
     \equiv 1\pmod{2}.
\end{align*}
Hence, for $n\geq1$, we have
\begin{align*}
    \opt(n) \equiv 0 \pmod{2}.\label{n2}
\end{align*}
The last congruence is reminiscent of the fact that
\[
\op(n)\equiv 0 \pmod 2, \quad \text{for all}~n>0.
\]
In this paper we prove several infinite families of Ramanujan-type congruences modulo powers of $2$ for the $\opt(n)$ function (Theorems \ref{thm:sellers-1}, \ref{thm:sellers-2}, \ref{thm:sellers-3} and \ref{thm:4}.)

As can be seen, there is nothing special about the definition of an overpartition triple, and it can be easily extended to an overpartition $k$-tuple, as was done by the third author with Keister and Vary \cite{KeisterSellersVary}. An overpartition $k$-tuple of $n$ is a $k$-tuple of overpartitions ($\xi_1, \xi_2, \ldots, \xi_k)$ such that the sum of the parts of $\xi_i$'s equals  $n$. The generating function for the number of overpartition $k$-tuples of $n$, denoted by $\op_k(n)$ is given by
\[
\sum_{n\geq 0}\op_k(n)q^n=\frac{f_{2}^k}{f_{1}^{2k}}.
\]
We can similarly define an overpartition $k$-tuple of $n$ with odd parts to be an overpartition $k$-tuple $(\xi_1, \xi_2, \ldots, \xi_k)$ of $n$ where all parts of $\xi_i$'s are odd. The generating function for the number of overpartition $k$-tuples of $n$ with odd parts, denoted by $\opt_k(n)$ is given by
\[
\sum_{n\geq 0}\opt_k(n)q^n=\frac{f_{2}^{3k}}{f_{1}^{2k}f_{4}^{k}}.
\]
The case $k=3$ here corresponds to overpartition tuples with odd parts. The case $k=2$ has also been studied, the interested reader can look at the work of Lin \cite{Lin}. The study of the arithmetic properties of the case $k=1$ was initiated by the third author and Hirschhorn \cite{HSellers}, who used the notation $\op_o(n)$ for $\opt_1(n)$, which we will use from now on.

In this paper we prove several Ramanujan-type congruences modulo powers of $2$ for the functions $\opt_{2k+1}(n)$ (for $k\geq 1$) (Theorems \ref{thm:k2}, \ref{thm:9}, \ref{Ram_congs_mod4} and \ref{thm:10-n}). We also give two infinite families of congruences modulo powers of $2$ for the $\opt_4(n)$ (Theorem \ref{inffammod16}) and $\opt_{2k+1}(n)$ (Theorem \ref{thm:inf2k+1}) functions.

The paper is organized as follows: in Section \ref{sec:thm} we state our main results (and prove three of them quickly), and in Section \ref{sec:prelim} we state some known results which will be used for our proofs. Theorems \ref{thm:sellers-1}, \ref{thm:sellers-2} and \ref{thm:sellers-3} are proved in Section \ref{sec:sellers-1}. Theorem \ref{thm:4} is then proved in Section \ref{sec:inf}, Theorem \ref{thm:10-n} is proved in Section \ref{sec:thm-10-n} and Theorems \ref{inffammod16} and \ref{thm:inf2k+1} are proved in Section  \ref{sec:inf2}. We end the paper with some concluding remarks in Section \ref{sec:conc}, which also includes Conjecture \ref{conj2i}, a counterpart of Theorem \ref{thm:10-n}.

\section{Main Results}\label{sec:thm}

The first set of results we prove in this paper are as follows. 

\begin{theorem}\label{thm:sellers-1}
    For all $n\geq 0, \alpha\geq 0$ with $n, \alpha \in \mathbb{Z}$, we have
    \begin{equation}\label{eq:sellers-1}
        \opt(2^\alpha(4n+3))\equiv 0 \pmod 4.
    \end{equation}
\end{theorem}

\begin{theorem}\label{thm:sellers-2}
    For all $n\geq 0, \alpha\geq 0$ with $n, \alpha \in \mathbb{Z}$, we have
    \begin{equation}\label{eq:sellers-2}
        \opt(2^\alpha(8n+5))\equiv 0 \pmod 8.
    \end{equation}
\end{theorem}

\begin{theorem}\label{thm:sellers-3}
    For all $n\geq 0, \alpha\geq 0$ with $n, \alpha \in \mathbb{Z}$, we have
    \begin{equation}\label{eq:sellers-3}
        \opt(2^\alpha(8n+7))\equiv 0 \pmod{16}.
    \end{equation}
\end{theorem}
\begin{remark}
Theorem \ref{thm:sellers-3} is an extension of a result of 
Drema and Saikia \cite[Eq. (18)]{DremaSaikia} who proved that, for all integers $n\geq 0,$
\[
\opt(16n+14)\equiv 0 \pmod 8.
\]
\end{remark}
\noindent Theorems \ref{thm:sellers-1}, \ref{thm:sellers-2} and \ref{thm:sellers-3} are proved in Section \ref{sec:sellers-1}.

Many mathematicians working in this area, including Drema and Saikia \cite[Theorems 2 and 3]{DremaSaikia} have studied infinite families of congruences. We present two such infinite families in the following theorem.
\begin{theorem}\label{thm:4}
If $p\geq3$ is a prime, then for all $n\geq0$, $\alpha\geq0$ and $\delta\geq0$ with $n, \alpha, \delta \in \mathbb{Z}$, we have
    \begin{align}\label{thminf4}
        \opt(2\cdot3^{2\alpha}\cdot p^{2\delta+1}(pn+t)+3^{2\alpha}\cdot p^{2(\delta+1)}) \equiv 0 \pmod{4},
    \end{align}
      where $t\in \{1, 2, \ldots, p-1\}$,
      
      If $p>3$ is a prime satisfying $\left(\frac{-2}{p}\right)=-1$, then for all $n\geq0$, $\alpha\geq0$ and $\delta\geq0$ with $n, \alpha, \delta \in \mathbb{Z}$, we have
    \begin{align}\label{thminf}
        \opt(8\cdot 3^{2\alpha}\cdot p^{2\delta+1}(pn+r)+3^{2\alpha}\cdot p^{2(\delta+1)}) \equiv 0 \pmod{8},
    \end{align}
    where $r\in \{1, 2, \ldots, p-1\}$.
\end{theorem}
\noindent We prove the above theorem in Section \ref{sec:inf}.

We now state some results for overpartition $k$-tuples with odd parts. The first such result is reminiscent of a result proved by the third author with Keister and Vary \cite[Theorem 8]{KeisterSellersVary}.
\begin{theorem}\label{thm:k2}
    Let $k=(2^m)r$ with $m>0$ and $r$ odd. Then for all integers $n\geq 1$ we have
    \[
    \opt_k(n)\equiv 0 \pmod{2^{m+1}}.
    \]
\end{theorem}

\noindent The proof is very short and we can give it immediately using the following lemma from Keister, Sellers and Vary \cite[Lemma 7]{KeisterSellersVary}.
\begin{lemma}\label{lem1}
Let $m$ be a nonnegative integer, then for all integers $n$ such that $1\leq n\leq 2^m$, we have
\[
\binom{2^m}{n}2^n\equiv 0 \pmod{2^{m+1}}.
\]
\end{lemma}

\begin{proof}[Proof of Theorem \ref{thm:k2}]
We have the following
\begin{align*}
  \sum_{n\geq 0}\opt_{k}(n)q^n&  =\prod_{i=1}^\infty \left(\frac{1+q^{2i+1}}{1-q^{2i+1}}\right)^{k}\\
  &=\left[\prod_{i=1}^\infty \left(\frac{1+q^{2i+1}}{1-q^{2i+1}}\right)^{2^m}\right]^r\\
  &=\left[\prod_{i=1}^\infty \left(1+\frac{2q^{2i+1}}{1-q^{2i+1}}\right)^{2^m}\right]^r.
\end{align*}
Now, using the binomial theorem and Lemma \ref{lem1} we can conclude
\[
 \sum_{n\geq 0}\opt_{k}(n)q^n\equiv 1\pmod{2^{m+1}}.
\]
This completes the proof.
\end{proof}

We now state three results for the general cases of $\opt_{2k+1}(n)$.
\begin{theorem}\label{thm:9}
    For all $n\geq 0$ and $k\geq 1$ with $n, k \in \mathbb{Z}$ we have
    \[
    \opt_{2k+1}(n)\equiv \op_o(n)\pmod4.
    \]
\end{theorem}
\begin{proof}
    We have
 \begin{align}
        \sum_{n\geq 0}\opt_{2k+1}(n)q^n
        &=\frac{f_{2}^{3(2k+1)}}{f_{1}^{2(2k+1)}f_{4}^{2k+1}} \nonumber \\
        &\equiv \frac{f_{2}^{6k}f_{2}^3}{f_{2}^{2k}f_{2}^{4k}f_{1}^2f_{4}} \pmod{4} \nonumber \\
        &\equiv \frac{f_{2}^{3}}{f_{1}^{2}f_{4}} \pmod{4} \label{cong2k+1}\\
        &\equiv \sum_{n\geq 0}\op_o(n) q^n \pmod{4}. \nonumber 
    \end{align}
\end{proof}

We can use Theorem \ref{thm:9} to prove an infinite family of Ramanujan--like congruences modulo 4 for the functions $\opt_{2k+1}.$
\begin{theorem} 
\label{Ram_congs_mod4}
For all $n\geq 0$, $k\geq 0$ with $n, k \in \mathbb{Z}$ and $p\geq 5$ prime, and all quadratic nonresidues $r$ modulo $p$ with $1\leq r\leq p-1$ we have
    $$
\opt_{2k+1}(2pn+R) \equiv 0 \pmod 4 ,
    $$
    where 
    $$
    R = \begin{cases}
			r, & \text{if $r$ is odd,}\\
            p+r, & \text{if $r$ is even.}
		 \end{cases}
   $$
\end{theorem}
\begin{proof}
From the work of Hirschhorn and the third author \cite[Theorem 1.1]{HSellers}, we know that, for all $n\geq 1,$ $\op_o(n) \equiv 0 \pmod{4}$ if and only if $n$ is neither a square nor twice a square.  So we simply need to show that $2pn+R$ as defined above is never a square and never twice a square.  
Note that $2pn+R$ is always odd by definition, so we see immediately that $2pn+R$ can never be twice a square.  
Next, we see that $2pn+R \equiv r \pmod{p}$ again from the definition of $R.$  And since $r$ is defined to be a quadratic nonresidue modulo $p$, we know that $r$ cannot be congruent to a square modulo $p.$  Thus, $2pn+R$ cannot equal a square.  This concludes the proof.  
\end{proof}

\begin{theorem}\label{thm:10-n}
 For all $n\geq 0$ and $k\geq 0$ with $n, k \in \mathbb{Z}$, we have
    \begin{align*}
    \opt_{2k+1}(8n+1)&\equiv 0 \pmod 2,\\
    \opt_{2k+1}(8n+2)&\equiv 0 \pmod 2,\\
    \opt_{2k+1}(8n+3)&\equiv 0 \pmod 4,\\
    \opt_{2k+1}(8n+4)&\equiv 0 \pmod 2,\\
    \opt_{2k+1}(8n+5)&\equiv 0 \pmod 8,\\
    \opt_{2k+1}(8n+6)&\equiv 0 \pmod 4,\\
    \opt_{2k+1}(8n+7)&\equiv 0 \pmod{16}.
    \end{align*}
\end{theorem}

\noindent Theorem \ref{thm:10-n} is proved in Section \ref{sec:thm-10-n}. The approach used to prove Theorem \ref{thm:10-n} is very similar to the one that was taken recently by the third author \cite{sellers-new} to prove a conjecture of the first author \cite{Saikia}.

We will also prove the following infinite family of congruences similar to \eqref{thminf}.
\begin{theorem}\label{inffammod16}
    If $p\geq3$ is a prime, then for all $n\geq0$, $\alpha\geq0$ and $\delta\geq0$ with $n, \alpha, \delta \in \mathbb{Z}$, we have
    \begin{align}
        \opt_4(8\cdot3^{2\alpha}\cdot p^{2\delta+1}(pn+t)+2\cdot3^{2\alpha}\cdot p^{2(\delta+1)}+3^{2\alpha}-1) \equiv 0 \pmod{16},
    \end{align}
      where $t\in \{1, 2, \ldots, p-1\}$.
\end{theorem}
Lastly, the following result generalizes \eqref{thminf} for $\opt_k(n)$, where $k$ is odd.
\begin{theorem}\label{thm:inf2k+1}
If $p\geq3$ is a prime, then for all $n\geq0$, $\alpha\geq0$, $\delta\geq0$ and $k\geq 1$ with $n,k, \alpha, \delta \in \mathbb{Z}$, we have
    \begin{align*}
        \opt_{2k+1}(2\cdot3^{2\alpha}\cdot p^{2\delta+1}(pn+t)+3^{2\alpha}\cdot p^{2(\delta+1)}) \equiv 0 \pmod{4},
    \end{align*}
      where $t\in \{1, 2, \ldots, p-1\}$.
\end{theorem}
\noindent Theorems \ref{inffammod16} and \ref{thm:inf2k+1} are proved in Section \ref{sec:inf2}.

\section{Preliminaries}\label{sec:prelim}

Ramanujan's general theta function $f(a,b)$ is defined by 
\begin{align*}
	f(a,b) = \sum_{k=-\infty}^{\infty} a^{\frac{k(k+1)}{2}}b^{\frac{k(k-1)}{2}}, \quad |ab| < 1.
\end{align*}
Two special cases of $f(a,b)$ are given by
\begin{align}
	\phi(q) &:= f(q,q) = \sum_{k=-\infty}^\infty q^{k^2}=1+2\sum_{n\geq 1}q^{n^2},\label{varphi}
\end{align}
and
\begin{align*}
    \psi(q):=f(q,q^3)=\sum_{k=0}^\infty q^{k(k+1)/2}.
\end{align*}
We recall the following results.
\begin{lemma}\label{lem3}\textup{\cite[p. 40, Entries 25(i) and 25(ii)]{bcb3}}
We have
	\begin{align}\label{phi-2-dissect}
		\phi(q) = \phi(q^4) + 2q \psi (q^8).
	\end{align}
\end{lemma}

\begin{theorem}\cite[Theorem 2.4]{HSellers}\label{hs-1}
We have
    \begin{align*}
        \sum_{n\geq0}\overline{p}_{o}(n)q^n = \phi(q)\phi(q^2)\phi(q^4)^2\phi(q^8)^4\cdots
    \end{align*}
\end{theorem}

Some known 2-, 3-dissections (see for example \cite[Lemmas 2 and 3]{matching}) are stated in the following lemma, which will be used subsequently.
\begin{lemma}\label{lem-diss}
We have
\begin{align}
f_{1}^2 &= \frac{f_{2}f_{8}^5}{f_{4}^2f_{16}^2} - 2 q \frac{f_{2} f_{16}^2}{f_{8}},\label{disf1^2}\\
\frac{1}{f_{1}^2} &= \frac{f_{8}^5}{f_{2}^5 f_{16}^2} + 2 q \frac{f_{4}^2 f_{16}^2 }{f_{2}^5 f_{8}},\label{dis1byf1^2}\\
f_{1}^4 &= \frac{f_{4}^{10}}{f_{2}^2f_{8}^4} - 4 q \frac{f_{2}^2 f_{8}^4}{f_{4}^2},\label{disf1^4}\\
\frac{1}{f_{1}^4} &= \frac{f_{4}^{14}}{f_{2}^{14} f_{8}^4} + 4 q \frac{f_{4}^2 f_{8}^4}{f_{2}^{10}},\label{dis1byf1^4}\\
f_{1}f_{2} &= \frac{f_{6}f_{9}^4}{f_{3}f_{18}^2} - qf_{9}f_{18} - 2q^2 \frac{f_{3}f_{18}^4}{f_{6}f_{9}^2},\label{disf1f2}\\
f_{1}^3 &= \frac{f_{6}f_{9}^6}{f_{3}f_{18}^3}-3qf_{9}^3+4q^3 \frac{f_{3}^2f_{18}^6}{f_{6}^2f_{9}^3}\label{disf1^3}.
\end{align}
\end{lemma}
The next lemma gives a $p$-dissection of $f_{1}$.  

\begin{lemma}\cite[Theorem 2.2]{CUI2013507}\label{pdis}
	For a prime $p > 3$, we have 
	\begin{align*}
		f_{1} = (-1)^{\frac{\pm p-1}{6}} q^{\frac{p^2 - 1}{24}} f_{p^2} + \sum_{k \neq \frac{\pm p-1}{6}, k = - \frac{p-1}{2} }^{\frac{p-1}{2}} (-1)^k q^{\frac{3k^2 + k}{2}} f\Bigg(-q^{\frac{3p^2+(6k+1)p}{2}},-q^{\frac{3p^2-(6k+1)p}{2}}\Bigg),
	\end{align*} 
	where 
	\begin{center}
		$\dfrac{\pm p-1}{6}=$ 
		$\begin{cases}
			\dfrac{p-1}{6}, & if ~p\equiv~ 1~ \pmod{6},\\
			\dfrac{-p-1}{6}, & if ~p\equiv~ -1~ \pmod{6}.
		\end{cases}$
	\end{center}
 Furthermore, for $\dfrac{-(p-1)}{2} \leq k \leq \dfrac{p-1}{2}$ and $k \neq \dfrac{\pm p-1}{6}$,
	\begin{align*}
		\dfrac{3k^2 + k}{2} \not\equiv \dfrac{p^2 - 1}{24} \pmod{p}.
		\end{align*}
	\end{lemma}
The following lemma gives a $p$-dissection of $f_{1}^3$.
\begin{lemma}\cite[Lemma 2.3]{zakirbaruah}
    If $p\geq3$ is a prime, then
    \begin{align}\label{f13dis}
        f_{1}^3=\sum_{k\neq\frac{p-1}{2}, k=0}^{p-1} (-1)^{k} q^{\frac{k(k+1)}{2}} \sum_{n=0}^{\infty} (-1)^n (2pn+2k+1) q^{pn\cdot\frac{pn+2k+1}{2}} + p (-1)^{\frac{p-1}{2}} q^{\frac{p^2-1}{8}} f_{p^2}^3.
    \end{align}
    Furthermore, if $k\neq\dfrac{p-1}{2}, 0\leq k \leq p-1$, then
    \begin{align*}
        \frac{k^2+k}{2} \not\equiv \frac{p^2-1}{8} \pmod{p}.
    \end{align*}
    
\end{lemma}

\section{Proofs of Theorems \ref{thm:sellers-1}, \ref{thm:sellers-2} and \ref{thm:sellers-3}}\label{sec:sellers-1}

We now provide proofs of Theorems \ref{thm:sellers-1}--\ref{thm:sellers-3} via elementary generating function manipulations.

\begin{proof}[Proof of Theorem \ref{thm:sellers-1}]
From \eqref{opt}, we recall
\begin{align}\label{n}
        \sum_{n\geq 0}\opt(n)q^n &=\frac{f_{2}^9}{f_{1}^6f_{4}^3}=  \frac{f_{2}^9}{f_{4}^3} \Bigg(\frac{f_{4}^{14}}{f_{2}^{14} f_{8}^4} + 4 q \frac{f_{4}^2 f_{8}^4}{f_{2}^{10}}\Bigg)\Bigg(\frac{f_{8}^5}{f_{2}^5 f_{16}^2} + 2 q \frac{f_{4}^2 f_{16}^2 }{f_{2}^5 f_{8}}\Bigg),
\end{align}
where we  have employed \eqref{dis1byf1^2} and \eqref{dis1byf1^4}. Now, working modulo $4$, we have
\begin{align}
   \sum_{n\geq 0}\opt(n)q^n & \equiv \frac{f_{2}^7}{f_{4}^3} \Bigg(\frac{f_{8}^5}{f_{2}^5 f_{16}^2} + 2 q \frac{f_{4}^2 f_{16}^2 }{f_{2}^5 f_{8}}\Bigg) \pmod{4}.\label{opt(n)}
\end{align}

Extracting the terms involving $q^{2n+1}$, dividing both sides by $q$ and replacing $q^2$ by $q$, we have
\begin{align}
    \sum_{n\geq 0}\opt(2n+1)q^n &\equiv 2\frac{f_{1}^2 f_{8}^2}{f_{2}f_{4}} \equiv 2 f_{2}^6 \pmod{4}\label{opt(2n+14)}.
\end{align}
Extracting the terms involving $q^{2n+1}$, we have
\begin{align}\label{sel-1}
    \opt(4n+3) \equiv 0 \pmod{4}.
\end{align}
This proves \eqref{eq:sellers-1} for the case $\alpha=0$.

Now, extracting the terms involving $q^{2n}$ from \eqref{opt(n)} and replacing $q^2$ by $q$, we have 
\begin{align}\label{2n4}
     \sum_{n\geq 0}\opt(2n)q^n &\equiv \frac{f_{1}^2f_{4}^5}{f_{2}^3f_{8}^2}
     \pmod{4}.
\end{align}

From equations \eqref{opt} and \eqref{2n4} we can conclude via elementary simplifications that, for all $n\geq 0$,
\begin{equation}\label{sel-3}
    \opt (2n) \equiv \opt(n)\pmod{4}.
\end{equation}
Finally, combining equations \eqref{sel-3} with \eqref{sel-1}, we have \eqref{eq:sellers-1}.
\end{proof}

\begin{proof}[Proof of Theorem \ref{thm:sellers-2}]
    Extracting the terms involving $q^{2n+1}$ from \eqref{n}, dividing both sides by $q$ and replacing $q^2$ by $q$ and working modulo 8, we have
\begin{align*}
     \sum_{n\geq 0}\opt(2n+1)q^n &\equiv 2\frac{f_{2}f_{8}^2}{f_{4}f_{1}^2} + 4 f_{4}f_{8}\equiv 2\frac{f_{2}f_{8}^2}{f_{4}}\Bigg(\frac{f_{8}^5}{f_{2}^5 f_{16}^2} + 2 q \frac{f_{4}^2 f_{16}^2 }{f_{2}^5 f_{8}}\Bigg) + 4 f_{4}f_{8} \pmod{8},
\end{align*}
thanks to (\ref{dis1byf1^2}).
Extracting the terms involving $q^{2n}$ and replacing $q^2$ by $q$, we have
\begin{align}
     \sum_{n\geq 0}\opt(4n+1)q^n &\equiv 2\frac{f_{4}^7}{f_{1}^4f_{2}f_{8}^2} + 4 f_{2}f_{4}\equiv 2\frac{f_{4}^7}{f_{2}f_{8}^2}\Bigg(\frac{f_{4}^{14}}{f_{2}^{14} f_{8}^4} + 4 q \frac{f_{4}^2 f_{8}^4}{f_{2}^{10}}\Bigg) + 4 f_{2}f_{4} \pmod{8}, \label{4n+18}
\end{align}
where in the last step, we  have employed \eqref{dis1byf1^4}.

Extracting the terms involving $q^{2n+1}$, dividing both sides by $q$ and replacing $q^2$ by $q$, we have
\begin{align*}
     \sum_{n\geq 0}\opt(8n+5)q^n &\equiv 8 \frac{f_{2}^9f_{4}^2}{f_{1}^{11}}\pmod{8}.
\end{align*}
Hence,
\begin{align}\label{eq:sel-21}
    \opt(8n+5) \equiv 0 \pmod{8}.
\end{align}
This proves \eqref{eq:sellers-2} for $\alpha=0$.

Now, extracting the terms involving $q^{2n}$ from \eqref{n}, replacing $q^2$ by $q$ and working modulo 8, we have
\begin{align}\label{eq:ext}
        \sum_{n\geq 0}\opt(2n)q^n &\equiv \frac{f_{2}^{11}f_{4}}{f_{1}^{10}f_{8}^2}\equiv 
        \frac{f_{2}^{7}f_{4}}{f_{1}^{2}f_{8}^2}\equiv
        \frac{f_{2}^7f_{4}}{f_{8}^2} \Bigg(\frac{f_{8}^5}{f_{2}^5 f_{16}^2} + 2 q \frac{f_{4}^2 f_{16}^2 }{f_{2}^5 f_{8}}\Bigg) \pmod{8}
\end{align}  
via (\ref{dis1byf1^2}).
Taking the odd dissection from the above, we obtain
\begin{eqnarray*}
\sum_{n\geq 0}\opt(4n+2)q^n 
& \equiv & 
2\frac{f_1^7f_2}{f_4^2}\cdot \frac{f_2^2f_{8}^2}{f_1^5f_4}  \equiv  
2\frac{f_1^2f_2^3f_8^2}{f_4^3} \pmod{8} \\
& \equiv & 
2\left( \frac{f_2f_8^5}{f_4^2f_{16}^2} - 2q\frac{f_2f_{16}^2}{f_8} \right)\frac{f_2^3f_8^2}{f_4^3} \pmod{8} 
\end{eqnarray*}
by \eqref{disf1^2}. 

Now we take the even dissection of the above to obtain 
\begin{eqnarray*}
\sum_{n\geq 0}\opt(8n+2)q^n 
& \equiv & 
2 \frac{f_1f_4^5}{f_2^2f_{8}^2} \cdot \frac{f_1^3f_4^2}{f_2^3}  \equiv  
2 \frac{f_1^4f_4^7}{f_2^5f_{8}^2}  \equiv 
2 \frac{f_2^2f_4^7}{f_2^5f_{8}^2}\equiv 
2 \frac{f_4^7}{f_2^3f_{8}^2} \pmod{8}.  \\
\end{eqnarray*}
This last expression is an even function of $q$.  Thus, for all $n\geq 0$, we know 
\begin{equation}\label{eqn:16n10mod8}
\opt(8(2n+1)+2) = \opt(16n+10)\equiv 0 \pmod{8}.
\end{equation}
This proves the $\alpha =1$ case of \eqref{eq:sellers-2}.

Extracting the terms involving even powers of $q$ from \eqref{eq:ext}, we have
\begin{align}\label{eq:p11}
        \sum_{n\geq 0}\opt(4n)q^n &\equiv \frac{f_{1}^{2}f_{2}f_{4}^3}{f_{8}^{2}}\equiv \frac{f_{2}f_{4}^3}{f_{8}^{2}} \Bigg(\frac{f_{2}f_{8}^5}{f_{4}^2f_{16}^2} - 2 q \frac{f_{2} f_{16}^2 }{f_{8}}\Bigg) \pmod{8}.
\end{align}
Extracting the terms involving $q^{2n+1}$, dividing both sides by $q$ and replacing $q^2$ by $q$, we have
\begin{align*}
      \sum_{n\geq 0}\opt(8n+4)q^n &\equiv -2\frac{f_{1}^{2}f_{2}^3f_{8}^2}{f_{4}^{3}}\equiv -2\frac{f_{2}^3f_{8}^2}{f_{4}^{3}} \Bigg(\frac{f_{2}f_{8}^5}{f_{4}^2f_{16}^2} - 2 q \frac{f_{2} f_{16}^2 }{f_{8}}\Bigg) \pmod{8}.
\end{align*}
Extracting the terms involving $q^{2n}$ from \eqref{n} and replacing $q^2$ by $q$, we have
\begin{align*}
     \sum_{n\geq 0}\opt(16n+4)q^n &\equiv -2\frac{f_{1}^{4}f_{4}^7}{f_{2}^{5}f_{8}^2}\equiv -2\frac{f_{4}^7}{f_{2}^{5}f_{8}^2} \Bigg(\frac{f_{4}^{10}}{f_{2}^2f_{8}^4} - 4 q \frac{f_{2}^2 f_{8}^4}{f_{4}^2}\Bigg)\pmod{8},
\end{align*}
where in the last step, we  have employed \eqref{disf1^4}.

Extracting the terms involving $q^{2n+1}$, dividing both sides by $q$ and replacing $q^2$ by $q$, we have
\begin{align*}
     \sum_{n\geq 0}\opt(32n+20)q^n &\equiv 8\frac{f_{2}^{5}f_{4}^2}{f_{1}^{3
}}\pmod{8}.
\end{align*}
Hence, for all $n\geq 0,$
\begin{align}\label{eq:sel-22}
    \opt(32n+20) \equiv 0 \pmod{8}.
\end{align}
This proves \eqref{eq:sellers-2} for $\alpha=2$.

Lastly, from equation \eqref{eq:p11} we have
\begin{align}
    \sum_{n\geq 0}\opt(8n)q^{2n} \equiv \frac{f_2^2f_4^3f_8^5}{f_8^2f_4^2f_{16}^2} \pmod 8 
\end{align} 
which means
\begin{align}
    \sum_{n\geq 0} \opt(8n)q^n\equiv \frac{f_{1}^{2}f_{2}f_{4}^3}{f_{8}^{2}} \pmod 8,
\end{align}
and thus we can conclude that, for all $n\geq 0$
\begin{equation}\label{eq:sel-23}
    \opt(8n)\equiv \opt(4n) \pmod 8
\end{equation}
thanks to (\ref{eq:p11}).
Combining equation \eqref{eq:sel-23} with \eqref{eq:sel-21}, \eqref{eqn:16n10mod8}, and \eqref{eq:sel-22}, we have \eqref{eq:sellers-2}.
\end{proof}
\begin{proof}[Proof of Theorem \ref{thm:sellers-3}]
    From \eqref{n}, we have
\begin{align*}
    \sum_{n=0}^{\infty}\opt(n)q^n &= \frac{f_{2}^9}{f_{1}^6f_{4}^3}\equiv \frac{f_{1}^{16}f_{2}}{f_{1}^6f_{4}^3}\equiv \frac{f_{2}}{f_{4}^3} \cdot f_{1}^2 \cdot \left(f_{1}^4\right)^2\\
    &\equiv \frac{f_{2}}{f_{4}^3}  \left(\frac{f_{4}^{10}}{f_{2}^2f_{8}^4} - 4 q \frac{f_{2}^2 f_{8}^4}{f_{4}^2}\right)^2 \left(\frac{f_{2}f_{8}^5}{f_{4}^2f_{16}^2} - 2 q \frac{f_{2} f_{16}^2 }{f_{8}}\right) \pmod{16}.
\end{align*}
Extracting the odd  powered terms of $q$, we have
\begin{align*}
    \sum_{n\geq 0}\opt(2n+1)q^n &\equiv -2\frac{f_2^{17}f_{8}^2}{f_1^2 f_4^9}-8\frac{f_1^2f_2^3f_4^5}{f_8^2}\pmod{16}. 
\end{align*}
With the aid of \eqref{dis1byf1^2}, we again extract the  odd powered terms of $q$ to arrive at
\begin{align*}
     \sum_{n\geq 0}\opt(4n+3)q^n &\equiv-4\frac{f_4 f_8^2}{f_2}\pmod{16}.
\end{align*}
Since the last congruence has no terms involving $q^{2n+1}$, we conclude that
\begin{align*}
    \opt(8n+7)\equiv0\pmod{16},
\end{align*}
which is the $\alpha=0$ case.

Extracting the terms involving even powers of $q$, we arrive at
\begin{align}
      \sum_{n\geq 0}\opt(2n)q^n &\equiv \frac{f_{2}^{15}}{f_{1}^2f_{4}^3f_{8}^2}\equiv \frac{f_{2}^{15}}{f_{4}^3f_{8}^2}  \Bigg(\frac{f_{8}^5}{f_{2}^5 f_{16}^2} + 2 q \frac{f_{4}^2 f_{16}^2 }{f_{2}^5 f_{8}}\Bigg) \pmod{16}.\label{2n16}
\end{align}
Extracting the terms involving $q^{2n+1}$, dividing both sides by $q$ and replacing $q^2$ by $q$, we have
\begin{align*}
      \sum_{n\geq 0}\opt(4n+2)q^n &\equiv 2\frac{f_{1}^2f_{2}^{3}f_{8}^2}{f_{4}^3}\equiv 2\frac{f_{2}^{3}f_{8}^2}{f_{4}^3} \Bigg(\frac{f_{2}f_{8}^5}{f_{4}^2f_{16}^2} - 2 q \frac{f_{2} f_{16}^2 }{f_{8}}\Bigg)\pmod{16}.
\end{align*}
Extracting the terms involving odd powers of $q$, we have
\begin{align}
      \sum_{n\geq 0}\opt(8n+6)q^n &\equiv -4\frac{f_{4}f_{8}^{2}}{f_{2}}\pmod{16}.
\end{align}
Extracting the terms involving $q^{2n+1}$, dividing both sides by $q$ and replacing $q^2$ by $q$, we have
\begin{align}\label{sel-31}
    \opt(16n+14) &\equiv 0 \pmod{16}.
\end{align}
This proves \eqref{eq:sellers-3} for $\alpha=1$.

Extracting the terms involving $q^{2n}$ from \eqref{2n16} and replacing $q^2$ by $q$, we have
\begin{align}
     \sum_{n\geq 0}\opt(4n)q^n &\equiv \frac{f_{1}^{10}f_{4}^{3}}{f_{2}^3f_{8}^2}\equiv \frac{f_{4}^{3}}{f_{2}^3f_{8}^2}\cdot(f_{1}^4)^2 \cdot f_{1}^2\label{eq:p13}\\&\equiv \frac{f_{4}^{3}}{f_{2}^3f_{8}^2}  \left(\frac{f_{4}^{10}}{f_{2}^2f_{8}^4} - 4 q \frac{f_{2}^2 f_{8}^4}{f_{4}^2}\right)^2 \left(\frac{f_{2}f_{8}^5}{f_{4}^2f_{16}^2} - 2 q \frac{f_{2} f_{16}^2 }{f_{8}}\right)\pmod{16}.\nonumber
\end{align}
Extracting the terms involving odd powers of $q$, we have
\begin{align*}
    \sum_{n\geq 0}\opt(8n+4)q^n &\equiv -8\frac{f_{2}^9f_{4}^3}{f_{1}^2f_{8}^2}-2\frac{f_{2}^{23}f_{8}^2}{f_{1}^6f_{4}^{11}}\\
    &\equiv  -8\frac{f_{2}^9f_{4}^3}{f_{2}f_{8}^2} - 2 \frac{f_{2}^{23}f_{8}^2}{f_{2}^4f_{4}^{11}}f_{1}^2\\
    &\equiv  -8\frac{f_{2}^8f_{4}^3}{f_{8}^2} - 2 \frac{f_{2}^{23}f_{8}^2}{f_{2}^4f_{4}^{11}}\Bigg(\frac{f_{2}f_{8}^5}{f_{4}^2f_{16}^2} - 2 q \frac{f_{2} f_{16}^2 }{f_{8}}\Bigg) \pmod{16}.
\end{align*}
Extracting the terms involving $q^{2n+1}$, dividing both sides by $q$ and replacing $q^2$ by $q$, we have
\begin{align*}
     \sum_{n\geq 0}\opt(16n+12)q^n &\equiv 4\frac{f_{1}^{20}f_{4}f_{8}^2}{f_{2}^{11}}\equiv 4\frac{f_{4}f_{8}^2}{f_{2}}\pmod{16}.
\end{align*}
Extracting the terms involving $q^{2n+1}$, dividing both sides by $q$ and replacing $q^2$ by $q$, we have
\begin{align}\label{sel-32}
    \opt(32n+28) \equiv 0 \pmod{16}.
\end{align}
This proves \eqref{eq:sellers-3} for $\alpha=2$.

From equation \eqref{eq:p13} with a little simplification we obtain
\begin{equation*}
    \sum_{n\geq 0}\opt(8n)q^{2n}\equiv \frac{f_4^{21}}{f_2^6f_8^5f_{16}^2} \pmod{16}.
\end{equation*}
Further simplifying and using \eqref{eq:p13} yields
\begin{equation}\label{sel-33}
    \opt(8n)\equiv \opt(4n) \pmod{16}.
\end{equation}
Combining equations \eqref{sel-33} with \eqref{sel-31} and \eqref{sel-32} we have \eqref{eq:sellers-3}.
\end{proof}

\section{Proof of Theorem \ref{thm:4}}\label{sec:inf}
\begin{proof}[Proof of \eqref{thminf4}]
From \eqref{opt(2n+14)}, we recall
\begin{align}
    \sum_{n\geq 0}\opt(2n+1)q^n  & \equiv 2 f_{2}^6\equiv 2f_{4}^3 \pmod{4}.\label{inf1mod4}
\end{align}
Employing \eqref{disf1^3} in the last step, extracting the terms involving $q^{3n+4}$, dividing both sides by $q^4$ and replacing $q^3$ by  $q$, we arrive at
\begin{align*}
      \sum_{n\geq 0}\opt(6n+9)q^n  & \equiv 2 f_{12}^3 \pmod{4}.
\end{align*}
Further extracting the terms involving $q^{3n}$ and replacing $q^{3}$ by $q$, we have
\begin{align}
     \sum_{n\geq 0}\opt(18n+9)q^n  & \equiv 2 f_{4}^3 \pmod{4}.\label{inf2mod4}
\end{align}
From \eqref{inf1mod4} and \eqref{inf2mod4}, we have
\begin{align*}
    \opt(2n+1) &\equiv \opt(18n+9)\\
    &\equiv \opt(9(2n+1))\\
    &\qquad\vdots\\
    &\equiv \opt(2\cdot 3^{2\alpha}n+ 3^{2\alpha})\pmod{4}
\end{align*}
for any $\alpha > 0$.
Hence for any $\alpha \geq 0$, we have
\begin{align}\label{induction1mod4}
    \sum_{n=0}^{\infty} \opt(2\cdot 3^{2\alpha}n+ 3^{2\alpha})q^n \equiv 2f_{4}^3 \pmod{4}.
\end{align}

Next, with the help of mathematical induction, we prove that if $p\geq3$ is a prime, then for $\delta\geq0$, we have
\begin{align}\label{infmod4}
        \sum_{n=0}^{\infty} \opt(2\cdot3^{2\alpha}\cdot p^{2\delta}n+3^{2\alpha}\cdot p^{2\delta}) q^n\equiv 2\cdot p^{\delta} (-1)^{\delta\left(\frac{p-1}{2}\right)}f_{4}^3\pmod{4},
    \end{align}
Clearly the case $\delta=0$ is  true by \eqref{induction1mod4}.

Using the $p$-dissection of $f_{1}^3$ stated in \eqref{f13dis}, we obtain
\begin{align*}
    &\sum_{n=0}^{\infty} \opt(2\cdot3^{2\alpha}\cdot p^{2\delta}n+3^{2\alpha}\cdot p^{2\delta}) q^n \notag\\&\equiv 2p^{\delta} (-1)^{\delta\left(\frac{p-1}{2}\right)}\Bigg[\sum_{k\neq\frac{p-1}{2}, k=0}^{p-1} (-1)^{k} q^{4\frac{k(k+1)}{2}} \sum_{0}^{\infty} (-1)^n (2pn+2k+1) q^{4pn\cdot\frac{pn+2k+1}{2}} \notag \\& \qquad + p (-1)^{\frac{p-1}{2}} q^{4\frac{p^2-1}{8}} f_{4p^2}^3\Bigg]\pmod{4},
\end{align*}
Extracting the terms involving $q^{pn+\frac{p^2-1}{2}}$, dividing both sides by $q^{\frac{p^2-1}{2}}$ and replacing $q^p$ by $q$, we arrive at
\begin{align}\label{inffammod4}
     \sum_{n=0}^{\infty} \opt(2\cdot3^{2\alpha}\cdot p^{2\delta+1}n+3^{2\alpha}\cdot p^{2(\delta+1)}) q^n\equiv 2\cdot p^{\delta+1} (-1)^{(\delta+1)\left(\frac{p-1}{2}\right)}f_{4p}^3\pmod{4},
\end{align}
Extracting the terms involving $q^{pn}$ and replacing $q^p$ by $q$, we obtain
\begin{align*}
     \sum_{n=0}^{\infty} \opt(2\cdot3^{2\alpha}\cdot p^{2(\delta+1)}n+3^{2\alpha}\cdot p^{2(\delta+1)}) q^n\equiv 2\cdot p^{\delta+1} (-1)^{(\delta+1)\left(\frac{p-1}{2}\right)}f_{4}^3\pmod{4},
\end{align*}
which is the $\delta+1$ case of \eqref{infmod4}. This completes the proof of \eqref{infmod4}.

Now, comparing the coefficients of $q^{pn+t}$ for $t\in\{1, 2, \cdots, p-1\}$ from \eqref{inffammod4}, we have
\begin{align*}
    \opt(2\cdot3^{2\alpha}\cdot p^{2\delta+1}(pn+t)+3^{2\alpha}\cdot p^{2(\delta+1)}) q^n\equiv 0\pmod{4},
\end{align*}
which completes the proof of \eqref{thminf4}.
\end{proof}

\begin{proof}[Proof of \eqref{thminf}]
From $\eqref{4n+18}$, we recall
\begin{align}
     \sum_{n\geq 0}\opt(4n+1)q^n &\equiv  2\frac{f_{4}^7}{f_{2}f_{8}^2}\Bigg(\frac{f_{4}^{14}}{f_{2}^{14} f_{8}^4} + 4 q \frac{f_{4}^2 f_{8}^4}{f_{2}^{10}}\Bigg) + 4 f_{2}f_{4} \pmod{8}.
\end{align}
Extracting the even powered terms of $q$, we obtain
\begin{align}
     \sum_{n\geq 0}\opt(8n+1)q^n &\equiv  2\frac{f_{2}^7}{f_{1}f_{4}^2}\cdot\frac{f_{2}^{14}}{f_{1}^{14} f_{4}^4}+ 4 f_{1}f_{2}\nonumber\\
     &\equiv 6f_{1}f_{2}\label{cong8n+18}\\ 
     &\equiv 6 \left(\frac{f_{6}f_{9}^4}{f_{3}f_{18}^2} - qf_{9}f_{18} - 2q^2 \frac{f_{3}f_{18}^4}{f_{6}f_{9}^2}\right)\pmod{8},\nonumber
\end{align}
where in the last step, we have employed \eqref{disf1f2}.\\
Extracting the terms involving $q^{3n+1}$, dividing both sides by $q$ and replacing $q^3$ by $q$, we obtain
\begin{align*}
     \sum_{n\geq 0}\opt(24n+9)q^n &\equiv -6f_{3}f_{6}\pmod{8}.
\end{align*}
Extracting the terms involving $q^{3n}$ and replacing $q^3$ by q, we obtain
\begin{align}
    \sum_{n\geq 0}\opt(72n+9)q^n &\equiv -6f_{1}f_{2}\pmod{8}.\label{1}
\end{align}
From \eqref{1} and \eqref{cong8n+18}, we arrive at
\begin{align*}
    \opt(8n+1) &\equiv -\opt(72n+9)\\
    &\equiv -\opt(8(3^2n+1)+1)\\
    &\equiv \opt(72(3^2n+1)+9)\\
    &\equiv \opt(8(3^4n+3^2\cdot 1+1)+1)\\
    &\qquad\vdots\notag\\
    &\equiv (-1)^{\alpha} \opt\left(8\left(3^{2\alpha}n+\frac{3^{2\alpha}-1}{8}\right)+1\right)\\
    &\equiv (-1)^{\alpha}\opt(8\cdot 3^{2\alpha}n+3^{2\alpha}) \pmod{8}
\end{align*}
for any $\alpha > 0$.
Hence, 
\begin{align}
     \sum_{n\geq 0}  (-1)^{\alpha}\opt(8\cdot 3^{2\alpha}n+3^{2\alpha}) q^{n} \equiv 6f_{1}f_{2}\pmod{8} \label{induction1mod8}
\end{align}
for any $\alpha > 0$.
Next, with the aid of mathematical induction, we prove that, if $p$ is a prime satisfying $\left(\frac{-2}{p}\right)=-1$, then for $\delta\geq0$,
\begin{align}\label{induction}
     \sum_{n\geq 0}  (-1)^{\alpha}\opt(8\cdot 3^{2\alpha}\cdot p^{2\delta}n+3^{2\alpha}\cdot p^{2\delta}) q^{n}\equiv (-1)^{\pm\delta\left(\frac{p-1}{3}\right)}6f_{1}f_{2}\pmod{8}.
\end{align}
Clearly, the case $\delta=0$ is true as \eqref{induction1mod8} holds.

Suppose that the result holds true for some $\delta>0$. Then, by using the $p$-dissection of $f_{1}$ stated in Lemma \ref{pdis}, we have 
\begin{align}\label{induction2mod8}
   & \sum_{n\geq 0}  (-1)^{\alpha}\opt(8\cdot 3^{2\alpha}\cdot p^{2\delta}n+3^{2\alpha}\cdot p^{2\delta})q^{n}\notag \\&\equiv 6\cdot (-1)^{\pm\delta\left(\frac{p-1}{3}\right)} \Bigg[\sum_{k \neq \frac{\pm p-1}{6}, k = - \frac{p-1}{2} }^{\frac{p-1}{2}} (-1)^k q^{\frac{3k^2 + k}{2}} f\Bigg(-q^{\frac{3p^2+(6k+1)p}{2}},-q^{\frac{3p^2-(6k+1)p}{2}}\Bigg)\\&\quad+(-1)^{\frac{\pm p-1}{6}} q^{\frac{p^2 - 1}{24}} f_{p^2}\Bigg] \cdot \Bigg[\sum_{k \neq \frac{\pm p-1}{6}, k = - \frac{p-1}{2} }^{\frac{p-1}{2}} (-1)^k q^{2\cdot\frac{3k^2 + k}{2}} f\Bigg(-q^{2\cdot\frac{3p^2+(6k+1)p}{2}},-q^{2\cdot\frac{3p^2-(6k+1)p}{2}}\Bigg)\\&\quad +(-1)^{\frac{\pm p-1}{6}} q^{2\cdot\frac{p^2 - 1}{24}} f_{2\cdot p^2} \Bigg]\pmod8.
\end{align}
Now we consider the congruence 
\begin{align}\label{cong-eqn1}
	\frac{3k^2+k}{2}+2\cdot\frac{3m^2+m}{2} \equiv \frac{p^2-1}{8} \pmod{p},
\end{align}
where $-\frac{p-1}{2} \leq k , m \leq \frac{p-1}{2}$. Since the above congruence is equivalent to solving the congruence 
\begin{align*}
	(6k+1)^2+2(6m+1)^2 \equiv 0 \pmod{p},
\end{align*}
and $\left(\frac{-2}{p}\right) = -1$, it follows that \eqref{cong-eqn1} has the unique solution $k = m = \frac{\pm p-1}{6}$. Therefore, extracting the terms involving $q^{pn+\frac{p^2-1}{8}}$ from both sides of \eqref{induction2mod8}, dividing by $q^{\frac{p^2-1}{8}}$ and replacing $q^p$ by $q$ we find that
\begin{align}\label{cong-infmod8}
     \sum_{n\geq 0}  (-1)^{\alpha}\opt(8\cdot 3^{2\alpha}\cdot p^{2\delta+1}n+3^{2\alpha}\cdot p^{2(\delta+1)}) q^{n}\equiv (-1)^{\pm(\delta+1)\left(\frac{p-1}{3}\right)}6f_{p}f_{2p}\pmod{8}.
\end{align}
Extracting the terms involving $q^{pn}$ and replacing $q^p$ by $q$, we arrive at
\begin{align*}
     \sum_{n\geq 0}  (-1)^{\alpha}\opt(8\cdot 3^{2\alpha}\cdot p^{2(\delta+1)}n+3^{2\alpha}\cdot p^{2(\delta+1)}) q^{n}\equiv (-1)^{\pm(\delta+1)\left(\frac{p-1}{3}\right)}6f_{1}f_{2}\pmod{8},
\end{align*}
which is the $\delta+1$ case of \eqref{induction}. Hence \eqref{induction} holds.

Now, equating the coefficients of $q^{pn+r}$ for $r\in\{1, 2, \cdots, p-1\}$ from both sides of \eqref{cong-infmod8}, we have \eqref{thminf}.
\end{proof}

\section{Proof of Theorem \ref{thm:10-n}}\label{sec:thm-10-n}

We first state and prove a result that will help us prove Theorem \ref{thm:10-n}.
\begin{lemma}\label{lem:7}
    For all odd $t\geq 1$, we have
\begin{align*}
    (\phi(q)\phi(q^2)\phi(q^4)^{2})^t = \sum_{j=0}^{7}a_{t,j}q^{j}F_{t,j}(q^8),
\end{align*}
where $F_{t,j}(q^8)$ is a function of $q^8$ whose power series representation has integer coefficients, and the following divisibilities hold:
\begin{align*}
    a_{t,1} &\equiv 0 \pmod{2},\\
    a_{t,2} &\equiv 0 \pmod{2},\\
    a_{t,3} &\equiv 0 \pmod{4},\\
    a_{t,4} &\equiv 0 \pmod{2},\\
    a_{t,5} &\equiv 0 \pmod{8},\\
    a_{t,6} &\equiv 0 \pmod{4},\\
    a_{t,7} &\equiv 0 \pmod{16}.
\end{align*}
\end{lemma}
\begin{proof}
We prove the above lemma with the help of induction on $t$. First, we prove the result for $t=1$.

 By multiple applications of Lemma \ref{lem3}, we have
    \begin{align*}
        \phi(q)\phi(q^2)\phi(q^4)^2 &= \left(2 q \psi(q^{8})+\phi (q^{16})+2 q^4 \psi(q^{32})\right)\\
        &\quad\times\left(2 q^2 \psi(q^{16})+\phi (q^{32})+2 q^{8} \psi(q^{64})\right)\\&\quad \times\left(2 q^{4} \psi(q^{32})+\phi (q^{64})+2 q^{16} \psi(q^{128})\right)^2.
    \end{align*}

   To expand the product, we set $p[n$\_$]:=2 q^{4 n} \psi _{q^{32 n}}+\phi _{q^{16 n}}+2 q^n \psi _{q^{8 n}}$ in Mathematica, where $\phi_{q^{n}}=\phi(q^n)$ and $\psi_{q^{n}}=\psi(q^n)$.
    Now, we expand this expression modulo 16 in Mathematica using the command PolynomialMod[p[1] p[2] $p[4]^2$,16] to arrive at
    \begin{align*}
        \phi(q)\phi(q^2)\phi(q^4)^2 &\equiv \phi(q^{16}) \phi(q^{32}) \phi(q^{64})^2 + 2 q \psi(q^8) \phi(q^{32}) \phi(q^{64})^2+2 q^2 \psi(q^{16}) \phi(q^{16}) \phi(q^{64})^2\\
        &\quad +4 q^3 \psi(q^8) \psi(q^{16}) \phi(q^{64})^2+4 q^4 \psi(q^{32}) \phi(q^{16}) \phi(q^{32}) \phi(q^{64})\\
        &\quad +2 q^4 \psi(q^{32}) \phi(q^{32}) \phi(q^{64})^2+8 q^5 \psi(q^8) \psi(q^{32}) \phi(q^{32}) \phi(q^{64})\\
        &\quad +8 q^6 \psi(q^{16}) \psi(q^{32}) \phi(q^{16}) \phi(q^{64})+4 q^6 \psi(q^{16}) \psi(q^{32}) \phi(q^{64})^2\\
        &\quad +2 q^8 \psi(q^{64}) \phi(q^{16}) \phi(q^{64})^2+4 q^8 \psi(q^{32})^2 \phi(q^{16}) \phi(q^{32})\\
        &\quad +8 q^8 \psi(q^{32})^2 \phi(q^{32}) \phi(q^{64}) +8 q^9 \psi(q^8) \psi(q^{32})^2 \phi(q^{32})\\
        &\quad +4 q^9 \psi(q^8) \psi(q^{64}) \phi(q^{64})^2 +8 q^{10} \psi(q^{16}) \psi(q^{32})^2 \phi(q^{16})\\
        &\quad+8 q^{12} \psi(q^{32})^3 \phi(q^{32}) +8 q^{12} \psi(q^{32}) \psi(q^{64}) \phi(q^{16}) \phi(q^{64})\\
        &\quad+4 q^{12} \psi(q^{32}) \psi(q^{64}) \phi(q^{64})^2 +8 q^{16} \psi(q^{32})^2 \psi(q^{64}) \phi(q^{16})\\
        &\quad+4 q^{16} \psi(q^{128}) \phi(q^{16}) \phi(q^{32}) \phi(q^{64}) +8 q^{17} \psi(q^8) \psi(q^{128}) \phi(q^{32}) \phi(q^{64})\\
        &\quad+8 q^{18} \psi(q^{16}) \psi(q^{128}) \phi(q^{16}) \phi(q^{64}) +8 q^{20} \psi(q^{32}) \psi(q^{128}) \phi(q^{16}) \phi(q^{32})\\
        &\quad+8 q^{20} \psi(q^{32}) \psi(q^{128}) \phi(q^{32}) \phi(q^{64}) +8 q^{24} \psi(q^{64}) \psi(q^{128}) \phi(q^{16}) \phi(q^{64})\\
        &\quad+4 q^{32} \psi(q^{128})^2 \phi(q^{16}) \phi(q^{32}) +8 q^{33} \psi(q^8) \psi(q^{128})^2 \phi(q^{32})\\
        &\quad+8 q^{34} \psi(q^{16}) \psi(q^{128})^2 \phi(q^{16}) +8 q^{36} \psi(q^{32}) \psi(q^{128})^2 \phi(q^{32})\\
        &\quad+8 q^{40} \psi(q^{64}) \psi(q^{128})^2 \phi(q^{16})\pmod{16},
    \end{align*}
    \ 

    \ 
    
    from which it is clear that the coefficients in the 8-dissection satisfy the criteria in the statement of the lemma.

    Next, we move to the induction step of the proof. Let the result be true for some $t\geq1$. Now,
    \begin{align*}
         (\phi(q)\phi(q^2)\phi(q^4)^{2})^{t+2} &=  (\phi(q)\phi(q^2)\phi(q^4)^{2})^t \times (\phi(q)\phi(q^2)\phi(q^4)^{2})^2\\
         & = \left(\sum_{j=0}^{7}a_{t,j}q^{j}F_{t,j}(q^8)\right)(\phi(q)\phi(q^2)\phi(q^4)^{2})^2.
    \end{align*}
Let us write
    \begin{align*}
        \sum_{j=0}^{7}a_{t,j}q^{j}F_{t,j}(q^8) &= b_{t,0}F_{t,0}(q^8)+2b_{t,1}qF_{t,1}(q^8)+2b_{t,2}q^2F_{t,2}(q^8)+4b_{t,3}q^3F_{t,3}(q^8)\\&\quad+2b_{t,4}q^4F_{t,4}(q^8)+8b_{t,5}q^5F_{t,5}(q^8)+4b_{t,6}q^6F_{t,6}(q^8)+16b_{t,7}q^7F_{t,7}(q^8),
    \end{align*}
    for some integers $b_{t,i}$, $i=1, 2, 3, 4, 5, 6, 7$. 

      Expanding modulo 16 similarly as in the last page using the Mathematica command\\ PolynomialMod$[(b_{t,0}F_{t,0}+2b_{t,1}qF_{t,1}+2b_{t,2}q^2F_{t,2}+4b_{t,3}q^3F_{t,3}+2b_{t,4}q^4F_{t,4}+8b_{t,5}q^5F_{t,5}+4b_{t,6}q^6F_{t,6}) \times p[1]^2 p[2]^2 p[4]^4, 16]$, where $F_{t,j}=F_{t,j}(q^8) (0\leq j\leq7)$, we have
\begin{align*}
    &\left(\sum_{j=0}^{7}a_{t,j}q^{j}F_{t,j}(q^8)\right)(\phi(q)\phi(q^2)\phi(q^4)^{2})^2\\ 
    &\equiv  8 q^{32}b_{t,0} F_{t,0}(q^8) \psi(q^{128})^2 \phi(q^{16})^2 \phi(q^{32})^2 \phi (q^{64})^2  + 8 q^{20 }b_{t,4} F_{t,4}(q^8) \psi(q^{64})^2 \phi(q^{16})^2 \phi(q^{64})^4 \\ &\quad+ 8 q^{18}b_{t,2} F_{t,2}(q^8) \psi(q^{64})^2 \phi(q^{16})^2 \phi(q^{64})^4  + 8 q^{17}b_{t,1} F_{t,1}(q^8) \psi(q^{64})^2 \phi(q^{16})^2 \phi(q^{64})^4 \\&\quad+ 4 q^{16}b_{t,0} F_{t,0}(q^8) \psi(q^{64})^2 \phi(q^{16})^2 \phi(q^{64})^4+8 q^{16} b_{t,0} F_{t,0}(q^8) \psi(q^{128}) \phi(q^{16})^2 \phi(q^{32})^2 \phi(q^{64})^3 \\&\quad+ 8 q^{12}b_{t,4} F_{t,4}(q^8) \psi(q^{32})^2 \phi(q^{32})^2 \phi(q^{64})^4+8 q^{12}b_{t,4} F_{t,4}(q^8) \psi(q^{64}) \phi(q^{16})^2 \phi(q^{32}) \phi(q^{64})^4\\ &\quad+ 8 q^{10}b_{t,2} F_{t,2}(q^8) \psi(q^{32})^2 \phi(q^{32})^2 \phi(q^{64})^4+8 q^{10}b_{t,0} F_{t,0}(q^8) \psi(q^{16}) \psi(q^{64}) \phi(q^{16})^2 \phi(q^{64})^4\\&\quad+8 q^{10}b_{t,2} F_{t,2}(q^8) \psi(q^{64}) \phi(q^{16})^2 \phi(q^{32}) \phi(q^{64})^4+8 q^9 b_{t,1} F_{t,1}(q^8) \psi(q^{32})^2 \phi(q^{32})^2 \phi(q^{64})^4 \\&\quad+8 q^9 b_{t,1} F_{t,1}(q^8) \psi(q^{64}) \phi(q^{16})^2 \phi(q^{32}) \phi(q^{64})^4 + 4 q^8 b_{t,0} F_{t,0}(q^8) \psi(q^{32})^2 \phi(q^{32})^2 \phi(q^{64})^4\\ &\quad+ 8 q^8 b_{t,4} F_{t,4}(q^8) \psi(q^{16})^2 \phi(q^{16})^2 \phi(q^{64})^4 +4 q^8 b_{t,0} F_{t,0}(q^8) \psi(q^{64}) \phi(q^{16})^2 \phi(q^{32}) \phi(q^{64})^4\\&\quad+8 q^8 b_{t,0} F_{t,0}(q^8) \psi(q^{32})^2 \phi(q^{16})^2 \phi(q^{32})^2 \phi(q^{64})^2 + 8 q^8 b_{t,4} F_{t,4}(q^8) \psi(q^{32}) \phi(q^{16}) \phi(q^{32})^2 \phi(q^{64})^4\\&\quad +  8 q^6 b_{t,2} F_{t,2}(q^8) \psi(q^{16})^2 \phi(q^{16})^2 \phi(q^{64})^4+8 q^6 b_{t,2} F_{t,2}(q^8) \psi(q^{32}) \phi(q^{16}) \phi(q^{32})^2 \phi(q^{64})^4\\&\quad+8 q^6 b_{t,4} F_{t,4}(q^8) \psi(q^{16}) \phi(q^{16})^2 \phi(q^{32}) \phi(q^{64})^4+4 q^6 b_{t,6} F_{t,6}(q^8) \phi(q^{16})^2 \phi(q^{32})^2 \phi(q^{64})^4\\&\quad+8 q^6 b_{t,4} F_{t,4}(q^8) \psi(q^8)^2 \phi(q^{32})^2 \phi(q^{64})^4+8 q^5 b_{t,1} F_{t,1}(q^8) \psi(q^{16})^2 \phi(q^{16})^2 \phi(q^{64})^4\\&\quad+8 q^5 b_{t,1} F_{t,1}(q^8) \psi(q^{32}) \phi(q^{16}) \phi(q^{32})^2 \phi(q^{64})^4+8 q^5b_{,5} F_{t,5}(q^8) \phi(q^{16})^2 \phi(q^{32})^2 \phi(q^{64})^4\\&\quad+8 q^5 b_{t,0} F_{t,0}(q^8) \psi(q^8) \psi(q^{32}) \phi(q^{32})^2 \phi(q^{64})^4+8 q^5 b_{t,4} F_{t,4}(q^8) \psi(q^8) \phi(q^{16}) \phi(q^{32})^2 \phi(q^{64})^4\\&\quad+4 q^4 b_{t,0} F_{t,0}(q^8) \psi(q^{16})^2 \phi(q^{16})^2 \phi(q^{64})^4+4q^4 b_{t,0} F_{t,0}(q^8) \phi(q^{16}) \phi(q^{32})^2\psi(q^{32}) \phi(q^{64})^4\\&\quad+8 q^4 b_{t,0} F_{t,0}(q^8) \psi(q^{32}) \phi(q^{16})^2 \phi(q^{32})^2 \phi(q^{64})^3+8 q^4b_{t,2} F_{t,2}(q^8) \psi(q^{16}) \phi(q^{16})^2 \phi(q^{32}) \phi(q^{64})^4\\&\quad+2q^4 b_{t,4} F_{t,4}(q^8) \phi(q^{16})^2 \phi(q^{32})^2 \phi(q^{64})^4+ 8 q^4b_{t,2} F_{t,2}(q^8) \psi(q^8)^2 \phi(q^{32})^2 \phi(q^{64})^4 \\&\quad+8q^3 b_{t,1} F_{t,1}(q^8) \psi(q^{16}) \phi(q^{16})^2 \phi(q^{32}) \phi(q^{64})^4+4 q^3 b_{t,3} F_{t,3}(q^8) \phi(q^{16})^2 \phi(q^{32})^2 \phi(q^{64})^4\\&\quad+8 q^3 b_{t,1} F_{t,1}(q^8) \psi(q^8)^2 \phi(q^{32})^2 \phi(q^{64})^4+8 q^3 b_{t,2} F_{t,2}(q^8) \psi(q^8) \phi(q^{16}) \phi(q^{32})^2 \phi(q^{64})^4\\&\quad+4 q^2 b_{t,0} F_{t,0}(q^8) \psi(q^{16}) \phi(q^{16})^2 \phi(q^{32}) \phi(q^{64})^4+2q^2 b_{t,2} F_{t,2}(q^8) \phi(q^{16})^2 \phi(q^{32})^2 \phi(q^{64})^4 \\&\quad+4q^2 b_{t.0} F_{t,0} (q^8)\psi(q^8)^2 \phi(q^{32})^2 \phi(q^{64})^4+8 q^2 b_{t,1} F_{t,1}(q^8) \psi(q^8) \phi(q^{16}) \phi(q^{32})^2 \phi(q^{64})^4\\&\quad+2 qb_{t,1} F_{t,1}(q^8) \phi(q^{16})^2 \phi(q^{32})^2 \phi(q^{64})^4+4 qb_{t,0} F_{t,0}(q^8) \psi(q^8) \phi(q^{16}) \phi(q^{32})^2 \phi(q^{64})^4\\&\quad+b_{t,0} F_{t,0}(q^8) \phi(q^{16})^2 \phi(q^{32})^2 \phi(q^{64})^4 \pmod{16},
    \end{align*}
which shows that the coefficients of $\left(\sum_{j=0}^{7}a_{t+2,j}q^{j}F_{t+2,j}(q^8)\right)$ satisfy the appropriate divisibilities.
\end{proof}
\begin{proof}[Proof of Theorem \ref{thm:10-n}]
    From Theorem \ref{hs-1} and the generating function for $\opt_k(n)$ we have, for any odd $t$,
    \begin{align*}
        \sum_{n\geq0} \opt_{t}(n)q^n = \phi(q)^t\phi(q^2)^t\phi(q^4)^{2t}\phi(q^8)^{4t}\cdots
    \end{align*}
    Since $\left(\prod_{i\geq3}\phi(q^{2^i})\right)^{2^{i-1}\cdot t}$ is a function of $q^8$, it is enough to do the 8-dissection of the first three terms. Also, as the highest modulus involved in the theorem is 16 and the other moduli are divisors of 16, we will prove our result modulo 16.

Hence,
 \begin{align*}
        \sum_{n\geq0} \opt_{t}(n)q^n &= \phi(q)^t\phi(q^2)^t\phi(q^4)^{2t}\phi(q^8)^{4t}\cdots\\
        &= \left(\sum_{j=0}^{7}a_{t,j}q^{j}F_{t,j}(q^8)\right) \left(\prod_{i\geq3}\phi(q^{2^i})\right)^{2^{i-1}\cdot t}.
    \end{align*}
The result now follows easily from Lemma \ref{lem:7}.
\end{proof}

\section{Proofs of Theorems \ref{inffammod16} and \ref{thm:inf2k+1}}\label{sec:inf2}

\begin{proof}[Proof of Theorem \ref{inffammod16}]
We have
\begin{align*}
    \sum_{n\geq0} \opt_{4}(n)q^n &= \frac{f_{2}^{12}}{f_{1}^8f_{4}^4}\equiv \frac{f_{2}^4f_{1}^8}{f_{4}^4}\equiv \frac{f_{2}^4}{f_{4}^4} \cdot \left(f_{1}^4\right)^2 \pmod{16}.
\end{align*}
Employing \eqref{dis1byf1^4} in the last step and extracting the odd powered terms of $q$, we have
\begin{align*}
    \sum_{n\geq0} \opt_{4}(2n+1)q^n &\equiv 8f_{32} \pmod{16}. 
\end{align*}
Further, extracting the odd powered terms of $q$, we obtain
\begin{align*}
      \sum_{n\geq0} \opt_{4}(4n+1)q^n &\equiv 8f_{16} \pmod{16}.
\end{align*}
Again, extracting the terms involving $q^{2n+1}$, dividing both sides by $q$ and replacing $q^2$ by $q$, we have
\begin{align*}
     \sum_{n\geq0} \opt_{4}(8n+1)q^n &\equiv 8f_{2}^3 \pmod{16}.
\end{align*}
Employing \eqref{disf1^3} in the last step, extracting the terms involving $q^{3n+2}$, dividing both sides by $q^2$ and replacing $q^3$ by $q$, we have
\begin{align*}
     \sum_{n\geq0} \opt_{4}(24n+17)q^n &\equiv 8f_{6}^3 \pmod{16}.
\end{align*}
Extracting the terms involving $q^{3n}$ and replacing $q^3$ by $q$, we have
\begin{align*}
     \sum_{n\geq0} \opt_{4}(72n+17)q^n &\equiv 8f_{2}^3 \pmod{16}.
\end{align*}
Hence,
\begin{align*}
    \opt_{4}(8n+1) \equiv \opt_{4}(72n+17)  \pmod{16}.
\end{align*}
The remainder of the proof is similar to the proof of \eqref{thminf4} in Section \ref{sec:inf}.
\end{proof}

\begin{proof}[Proof of Theorem \ref{thm:inf2k+1}]
 From 
 \eqref{cong2k+1}, recall 
    \begin{align*}
        \sum_{n\geq 0}\opt_{2k+1}(n)q^n& \equiv \frac{f_{2}^{3}}{f_{1}^{2}f_{4}}\pmod{4}. 
    \end{align*}
    Employing \eqref{dis1byf1^2} in the last step, extracting the odd  powered terms of $q$, we have
\begin{align*}
     \sum_{n\geq 0}\opt_{2k+1}(2n+1)q^n& \equiv 2\frac{f_2 f_{8}^2}{f_1^2 f_4}\nonumber\equiv 2 f_{4}^3 \pmod{4}. 
\end{align*}
The rest of the proof is exactly  similar to the  proof of \eqref{thminf4} in Section \ref{sec:inf}.
\end{proof}

\section{Concluding Remarks}\label{sec:conc}

\begin{enumerate}
    \item Keister, Sellers and Vary \cite[Theorems 6 and 9]{KeisterSellersVary}, and Chen \cite{Chen} proved some other results about overpartition $k$-tuples. It would be interesting to see if we can find any analogous results for overpartition $k$-tuples with odd parts.
    \item Chen \cite{zChen} and Kim \cite{Kim} have proved some congruences for overpartitions with odd parts and overpartition pairs respectively, using the theory of modular forms. It would be interesting to explore similar techniques for overpartition $k$-tuples with odd parts.
    \item Based on numerical evidence, we conjecture the following counterpart of Theorem \ref{thm:10-n}.
\begin{conjecture}\label{conj2i}
    For all $i\geq1$, $n\geq 0$ with $n, i \in \mathbb{Z}$ and odd $r$, we have
    \begin{align*}
        \opt_{2^ir}(8n+1)&\equiv0\pmod{2^{i+1}},\\
        \opt_{2^ir}(8n+2)&\equiv0\pmod{2^{2i+1}},\\
        \opt_{2^ir}(8n+3)&\equiv0\pmod{2^{i+3}},\\
        \opt_{2^ir}(8n+4)&\equiv0\pmod{2^{2i+4}},\\
        \opt_{2^ir}(8n+5)&\equiv0\pmod{2^{i+2}},\\
        \opt_{2^ir}(8n+6)&\equiv0\pmod{2^{2i+3}},\\
        \opt_{2^ir}(8n+7)&\equiv0\pmod{2^{i+4}}.
    \end{align*}
\end{conjecture}
\end{enumerate}
\begin{remark}
    We note that the first congruence of Conjecture \ref{conj2i} follows from Theorem \ref{thm:k2}. The fifth congruence of Conjecture \ref{conj2i} follows from a result analogous to a result of Keister, Vary and the third author \cite[Theorem 9]{KeisterSellersVary}. We do not prove the result here, the interested reader can verify the details. This means, only five congruences of Conjecture \ref{conj2i} remain, and we do not have a unified proof of them, like we have for the congruences in Theorem \ref{cong2k+1}.
\end{remark}
\textbf{Note added in proof.} After this paper was submitted, three cases of Conjecture \ref{conj2i} were proved completely, and several specific cases were proved by Das, Saikia, and Sarma \cite{DSS}.

\section*{Acknowledgements}

The second author would like to thank Prof. Nayandeep Deka Baruah for his encouragement and support. The second author was partially supported by an institutional fellowship for doctoral research from Tezpur University, Assam, India. He thanks the funding institution. The authors thank the anonymous referee for helpful comments.

\end{document}